\documentclass[%
  letterpaper,
  onecolumn,
  colorlinks,
]{preprint}

\usepackage[english]{babel}

\usepackage{amsmath}
\usepackage{amssymb}
\usepackage{amsthm}

\usepackage[linesnumbered, lined, algoruled]{algorithm2e}

\usepackage{graphicx}
\usepackage{subcaption}

\usepackage{tikz}
\usepackage{pgfplots}
  \pgfplotsset{compat = 1.17}
  \pgfkeys{/pgf/number format/.cd, 1000 sep = {\,}}
  \usepgfplotslibrary{external}
  \tikzset{external/system call = {%
    pdflatex \tikzexternalcheckshellescape
      -halt-on-error
      -interaction=batchmode
      -jobname "\image" "\texsource"}}
  \tikzexternaldisable

\newcommand{\C}{\ensuremath{\mathbb{C}}}
\newcommand{\R}{\ensuremath{\mathbb{R}}}

\newcommand{\trans}{\ensuremath{\mkern-1.5mu\mathsf{T}}}
\newcommand{\herm}{\ensuremath{\mathsf{H}}}
\DeclareMathOperator{\frob}{F}
\DeclareMathOperator{\ds}{d}

\newcommand{\CH}{\ensuremath{\mathcal{H}}}
\newcommand{\CI}{\ensuremath{\mathcal{I}}}

\newcommand{\bA}{\ensuremath{\boldsymbol{A}}}
\newcommand{\bB}{\ensuremath{\boldsymbol{B}}}
\newcommand{\bC}{\ensuremath{\boldsymbol{C}}}
\newcommand{\bD}{\ensuremath{\boldsymbol{D}}}
\newcommand{\bE}{\ensuremath{\boldsymbol{E}}}
\newcommand{\bG}{\ensuremath{\boldsymbol{G}}}
\newcommand{\bI}{\ensuremath{\boldsymbol{I}}}
\newcommand{\bJ}{\ensuremath{\boldsymbol{J}}}
\newcommand{\bL}{\ensuremath{\boldsymbol{L}}}
\newcommand{\bP}{\ensuremath{\boldsymbol{P}}}
\newcommand{\bQ}{\ensuremath{\boldsymbol{Q}}}
\newcommand{\bR}{\ensuremath{\boldsymbol{R}}}
\newcommand{\bT}{\ensuremath{\boldsymbol{T}}}
\newcommand{\bU}{\ensuremath{\boldsymbol{U}}}
\newcommand{\bV}{\ensuremath{\boldsymbol{V}}}
\newcommand{\bW}{\ensuremath{\boldsymbol{W}}}

\newcommand{\bj}{\ensuremath{\boldsymbol{j}}}
\newcommand{\bk}{\ensuremath{\boldsymbol{k}}}
\newcommand{\bu}{\ensuremath{\boldsymbol{u}}}
\newcommand{\bx}{\ensuremath{\boldsymbol{x}}}
\newcommand{\by}{\ensuremath{\boldsymbol{y}}}
\newcommand{\bz}{\ensuremath{\boldsymbol{z}}}

\newcommand{\bPhi}{\ensuremath{\boldsymbol{\Phi}}}
\newcommand{\bSigma}{\ensuremath{\boldsymbol{\Sigma}}}

\newcommand{\bAr}{\ensuremath{\skew4\widehat{\bA}}}
\newcommand{\bBr}{\ensuremath{\skew4\widehat{\bB}}}
\newcommand{\bCr}{\ensuremath{\skew4\widehat{\bC}}}
\newcommand{\bEr}{\ensuremath{\skew4\widehat{\bE}}}
\newcommand{\bGr}{\ensuremath{\skew4\widehat{\bG}}}

\newcommand{\bxr}{\ensuremath{\skew2\hat{\bx}}}
\newcommand{\byr}{\ensuremath{\skew2\hat{\by}}}

\newcommand{\tbA}{\ensuremath{\skew4\widetilde{\bA}}}
\newcommand{\tbB}{\ensuremath{\skew4\widetilde{\bB}}}
\newcommand{\tbC}{\ensuremath{\skew4\widetilde{\bC}}}
\newcommand{\tbE}{\ensuremath{\skew4\widetilde{\bE}}}
\newcommand{\tbG}{\ensuremath{\skew4\widetilde{\bG}}}

\newcommand{\AL}{\ensuremath{\mathbb{A}_{\operatorname{L}}}}
\newcommand{\BL}{\ensuremath{\mathbb{B}_{\operatorname{L}}}}
\newcommand{\CL}{\ensuremath{\mathbb{C}_{\operatorname{L}}}}
\newcommand{\EL}{\ensuremath{\mathbb{E}_{\operatorname{L}}}}

\renewcommand{\i}{\ensuremath{\mathfrak{i}}}
\newcommand{\bRcheck}{\ensuremath{\boldsymbol{\check{R}}}}
\newcommand{\bLcheck}{\ensuremath{\boldsymbol{\check{L}}}}
\newcommand{\bUcheck}{\ensuremath{\boldsymbol{\check{U}}}}
\newcommand{\bVcheck}{\ensuremath{\boldsymbol{\check{V}}}}
\newcommand{\bSigmacheck}{\ensuremath{\boldsymbol{\check{\Sigma}}}}
\newcommand{\bzero}{\ensuremath{\boldsymbol{0}}}

\newcommand{\tbz}{\ensuremath{\skew2\tilde{\bz}}}
\newcommand{\bJr}{\ensuremath{\bJ_{\operatorname{r}}}}
\newcommand{\bJl}{\ensuremath{\bJ_{\operatorname{\ell}}}}
\newcommand{\zMin}{\ensuremath{z}_{\operatorname{min}}}
\newcommand{\zMax}{\ensuremath{z}_{\operatorname{max}}}
\newcommand{\relerr}{\ensuremath{\operatorname{relerr}}}

\newcommand{\Quadbt}{\textsf{QuadBT}}
\newcommand{\symQuadbt}{\textsf{SymQuadBT}}
\newcommand{\FDsymQuadbt}{\textsf{FDSymQuadBT}}

\theoremstyle{plain}\newtheorem{theorem}{Theorem}
\theoremstyle{plain}\newtheorem{corollary}{Corollary}

\definecolor{matlabblue}{HTML}{0072BD}
\definecolor{matlaborange}{HTML}{D95319}
\definecolor{matlabyellow}{HTML}{EDB120}
\definecolor{matlabpurple}{HTML}{7E2F8E}
\definecolor{matlabgreen}{HTML}{77AC30}
\definecolor{matlablightblue}{HTML}{4DBEEE}
\definecolor{matlabred}{HTML}{A2142F}

\tikzstyle{sline} = [
  solid,
  line width = 1.5pt
]

\newcommand{\plotfontsize}{\footnotesize}

\pgfplotscreateplotcyclelist{plotlist}{
  {black, sline},
  {matlabblue, sline, dashed, mark options = {solid}, mark = o, mark repeat = 125, mark phase = 0},
  {matlabpurple, sline, dashdotted, mark options = {solid}, mark = triangle*, mark repeat = 125, mark phase = 15},
  {matlabgreen, sline, dotted, mark options = {solid}, mark = *, mark repeat = 125, mark phase = 30},
  {matlaborange, sline, dashed, mark options = {solid}, mark = diamond*, mark repeat = 125, mark phase = 45},
  {matlabred, sline, dashdotted, mark options = {solid}, mark = square*, mark repeat = 125, mark phase = 60},
  {matlabyellow, sline, dotted, mark options = {solid}, mark = +, mark repeat = 125, mark phase = 75},
}

\begin{document}


\title{Symmetric Hermite quadrature-based balanced truncation for learning
  linear dynamical systems from derivative data}

\author[$\ast$]{Sean Reiter}
\affil[$\ast$]{Courant Institute School of Mathematics, Computing, and
  Data Science, New York University, New York, NY 10012, USA.\authorcr
  \email{s.reiter@nyu.edu}, \orcid{0000-0002-7510-1530}}

\author[$\dagger$]{Steffen W. R. Werner}
\affil[$\dagger$]{Department of Mathematics,
  Division of Computational Modeling and Data Analytics, and
  National Security Institute,
  Virginia Tech, Blacksburg, VA 24061, USA.\authorcr
  \email{steffen.werner@vt.edu}, \orcid{0000-0003-1667-4862}}

\shorttitle{Symmetric Hermite quadrature-based balanced truncation}
\shortauthor{S. Reiter, S. W. R. Werner}
\shortdate{2026-05-29}
\shortinstitute{}



\abstract{%
  Data-driven reduced-order modeling is an essential component in the
  computer-aided design of control systems.
  In this work, we present a novel symmetric Hermite formulation of the
  quadrature-based balanced truncation algorithm that constructs linear
  reduced-order models from evaluations of the full-order system's 
  transfer function and its derivative.
  Significantly, the Hermite formulation preserves desirable qualitative
  properties of the system used to generate the data, such as
  state-space Hermiticity and, consequently, asymptotic stability.
}

\novelty{}

\maketitle


\section{Introduction}%
\label{sec:intro}

The data-driven modeling of dynamical processes is a crucial tool in
the computational sciences.
It enables the construction of high-fidelity mathematical models for
complex physical phenomena when first-principle models are not available,
but ample input/output data exist.
A fundamental system class of interest for the modeling of dynamical processes 
is that of linear continuous-time input-output systems.
In state-space form, such systems are expressed as
\begin{equation} \label{eqn:sys}
  \bE \dot{\bx}(t) = \bA \bx(t) + \bB \bu(t), \quad
  \by(t) = \bC \bx(t),
\end{equation}
with the matrices $\bE, \bA \in \C^{n \times n}$, $\bB \in \C^{n \times m}$,
and $\bC \in \C^{p \times n}$.
Therein, the external inputs $\bu\colon \R \to \C^{m}$ are used to reach desired
internal states $\bx\colon \R \to \C^{n}$, which can be observed via the outputs
$\by\colon \R \to \C^{p}$.
Equivalently, the input-output behavior of~\cref{eqn:sys} is described in
the frequency domain by the corresponding transfer function
\begin{equation} \label{eqn:tf}
  \bG\colon \C \to \C^{p \times m} \quad\text{with}\quad
  \bG(s) = \bC (s \bE - \bA)^{-1} \bB.
\end{equation}
In this work, we consider the task of modeling linear dynamical systems
of the form~\cref{eqn:sys} from data.
Typical data for this task are transfer function values
and the corresponding evaluation points
\begin{equation} \label{eqn:data}
  \left\{ \big(\i \omega_{i}, \bG(\i \omega_{i}) \big) \right\}_{i = 1}^{N},
\end{equation}
with real frequencies $\omega_{i} \in \R$, for $i = 1, \ldots, N$.

A variety of methods has been developed for the construction of systems such
as~\cref{eqn:sys} from frequency-response data~\cref{eqn:data} in recent
years, for example, those based on the construction of rational
interpolants or the solution of nonlinear optimization problems.
For a comprehensive overview, we refer to~\cite{AntBG20, BenSGetal21} and the
references therein.
One such approach is the recently developed quadrature-based balanced
truncation (QuadBT) method~\cite{GosGB22}.
QuadBT is a data-driven reformulation of the classical balanced truncation
model order reduction technique~\cite{Moo81, MulR76} that enables the
construction of (approximately) balanced low-dimensional models directly from
input-output frequency-response data of the form~\cref{eqn:data}.

In this work, we introduce a symmetric Hermite formulation of the QuadBT
method, which allows the construction of linear con\-tin\-u\-ous-time systems
from data of the form
\begin{equation} \label{eqn:Hdata}
  \left\{ \big(\i \omega_{i}, \bG(\i \omega_{i}), \bG'(\i \omega_{i})
    \big) \right\}_{i = 1}^{N},
\end{equation}
where $\bG'(\sigma) := \left. \tfrac{\ds}{\ds s} \bG(s)
\right\rvert_{s = \sigma}$
denotes the evaluation of the first derivative of the transfer at the
point $\sigma \in \C$.
In contrast to the previous approach from~\cite{GosGB22}, our Hermite
formulation enables the use of symmetric quadrature formulas and allows to
preserve state-space Hermiticity in the constructed reduced-order
model.
As we will show, this also allows to preserve asymptotic stability in the
learned model if the given data are generated by a state-space Hermitian,
dissipative system.


\section{Balanced reduced-order models and quadrature approximations}%
\label{sec:prelims}

The QuadBT method~\cite{GosGB22} is based on the balanced truncation
(BT) model order reduction approach~\cite{Moo81, MulR76}, which constructs
reduced-order models by balancing the controllability and observability
Gramians of the system~\cref{eqn:sys} and subsequently truncating states
corresponding to the smallest magnitude eigenvalues of the balanced Gramians.
For a linear system of the form~\cref{eqn:sys} with $\bE$ invertible and
the eigenvalues of the matrix pencil $\lambda \bE - \bA$ being in the open left
half-plane, the controllability Gramian $\bP \in \C^{n \times n}$ and
observability Gramian $\bE^{\herm} \bQ \bE \in \C^{n \times n}$ can be defined
via the indefinite frequency-domain integrals
\begin{subequations} \label{eqn:Gramians}
\begin{align} 
  \bP & = \int\limits_{-\infty}^{\infty} (\i \omega \bE - \bA) \bB
    \bB^{\herm} (\i \omega \bE - \bA)^{\herm} \ds \omega
  \quad\text{and} \\
  \bQ & = \int\limits_{-\infty}^{\infty} (\i \omega \bE - \bA)^{\herm} \bC^{\herm}
    \bC (\i \omega \bE - \bA) \ds \omega.
\end{align}
\end{subequations}
Given the Cholesky-like factors of the Gramian matrices
$\bP = \bR \bR^{\herm}$ and $\bQ = \bL \bL^{\herm}$, with
$\bR \in \C^{n \times n}$ and $\bL \in \C^{n \times n}$, and the singular value
decomposition
\begin{equation*}
  \bL^{\herm} \bE \bR = \begin{bmatrix} \bU_{1} & \bU_{2} \end{bmatrix}
    \begin{bmatrix} \bSigma_{1} & \bzero \\ \bzero & \bSigma_{2} \end{bmatrix}
    \begin{bmatrix} \bV_{1}^{\herm} \\[2pt] \bV_{2}^{\herm} \end{bmatrix},
  \quad\text{with}\quad
  \bSigma_{1} \in \R^{r \times r},
  \bU_{1} \in \C^{n \times r},
  \bV_{1} \in \C^{n \times r},
\end{equation*}
a balanced reduced-order model of the form
\begin{equation} \label{eqn:rom}
  \bEr \dot{\bxr}(t) = \bAr \bxr(t) + \bBr \bu(t), \quad
  \byr(t) = \bCr \bxr(t),
\end{equation}
with $\bEr, \bAr \in \C^{r \times r}$, $\bBr \in \C^{r \times m}$,
$\bCr \in \C^{p \times r}$, and $r \leq n$, is computed via
\begin{align*}
  \bEr & = \bI_{r}, \\
  \bAr & = \bSigma_{1}^{-1/2} \bU_{1}^{\herm} (\bL^{\herm} \bA \bR)
    \bV_{1} \bSigma_{1}^{-1/2}, \\
  \bBr & = \bSigma_{1}^{-1/2} \bU_{1}^{\herm} (\bL^{\herm} \bB), \\
  \bCr & = (\bC \bR) \bV_{1} \bSigma_{1}^{-1/2}.
\end{align*}

QuadBT is derived from the classical balanced truncation model order reduction
algorithm based on the following two observations:
\begin{enumerate}
  \item The reduced-order model~\cref{eqn:rom} is completely specified by
    the matrix products $\bL^{\herm} \bE \bR$, $\bL^{\herm} \bA \bR$,
    $\bL^{\herm} \bB$, and $\bC \bR$.
  \item These matrix products can be expressed in terms of transfer function
    data~\cref{eqn:data} by replacing the exact square-root factors
    $\bR$ and $\bL$ with (approximate) low-rank factors derived from
    numerical quadrature rules applied to the Gramians~\cref{eqn:Gramians}.
\end{enumerate}
Specifically, let the Gramian matrices in~\cref{eqn:Gramians} be approximated
via quadrature rules of the form
\begin{subequations} \label{eqn:QuadRules}
\begin{align} 
  \bP & \approx \sum\limits_{j = 1}^{J} \varrho_{j}^{2} 
    (\i \zeta_{j} \bE - \bA) \bB \bB^{\herm} (\i \zeta_{j} \bE - \bA)^{\herm}
  \quad\text{and} \\
  \bQ & \approx \sum\limits_{k = 1}^{K} \varphi_{k}^{2}
    (\i \vartheta_{k} \bE - \bA)^{\herm} \bC^{\herm}
    \bC (\i \vartheta_{k} \bE - \bA),
\end{align}
\end{subequations}
with the quadrature nodes and weights
$\{ (\zeta_{j}, \varrho_{j}) \}_{j = 1}^{J}$ and
$\{ (\vartheta_{k}, \varphi_{k}) \}_{k = 1}^{K}$.
Then, Cholesky-like factorizations of the quadrature-based approximations
to the Gramians $\bP \approx \bRcheck \bRcheck^{\herm}$ and
$\bQ \approx \bLcheck \bLcheck^{\herm}$ are given via
\begin{subequations} \label{eqn:QuadFactors}
\begin{align} 
  \bRcheck & =
    \begin{bmatrix} \varrho_{1} \bPhi(\i \zeta_{1})^{-1} \bB &
    \varrho_{2} \bPhi(\i \zeta_{2})^{-1} \bB & \ldots &
    \varrho_{J} \bPhi(\i \zeta_{J})^{-1} \bB \end{bmatrix}
  \quad\text{and}\\
  \bLcheck^{\herm} & =
    \begin{bmatrix} \varphi_{1} \bC \bPhi(\i \vartheta_{1})^{-1} \\
      \varphi_{2} \bC \bPhi(\i \vartheta_{2})^{-1}\\[2pt]
      \vdots\\[2pt]
      \varphi_{K} \bC \bPhi(\i \vartheta_{K})^{-1}
    \end{bmatrix},
\end{align}
\end{subequations}
where $\bPhi(s) = s \bE - \bA$.
With these particular Gramian factors, the matrix products in balanced
truncation then take the form
\begin{subequations} \label{eqn:DataMatrices}
\begin{align} 
  \EL & = \bLcheck^{\herm} \bE \bRcheck \in \C^{pK \times mJ}, \\
  \AL & = \bLcheck^{\herm} \bA \bRcheck \in \C^{pK \times mJ}, \\
  \BL & = \bLcheck^{\herm} \bB \in \C^{pK \times m}, \\
  \CL & = \bC \bRcheck \in \C^{p \times mJ}.
\end{align}
\end{subequations}
It has been shown in~\cite{GosGB22} that these matrices $\EL$, $\AL$,
$\BL$, and $\CL$ can be directly written in terms of transfer function data of
the form~\cref{eqn:data}.
The rest of the QuadBT approach follows the classical balanced truncation
methodology by truncating the states corresponding to the small singular values
of the $\EL$ matrix.


\section{Symmetric Hermite quadrature-based balanced truncation}%
\label{sec:method}

A limitation of the formulas from~\cite{GosGB22} that provide the construction
of the data matrices $\EL$, $\AL$, $\BL$, and $\CL$ in~\cref{eqn:DataMatrices}
is that they can only be applied under the assumption that the quadrature rules
used to derive $\bRcheck$ and $\bLcheck$ are asymmetric and have distinct
nodes, i.e., $\zeta_{j} \neq \vartheta_{k}$ for all $j = 1, \ldots, J$ and
$k = 1, \ldots, K$.
As a consequence, the data matrices in~\cref{eqn:DataMatrices} cannot be
symmetric, thereby preventing the preservation of potential symmetry structures
in the reduced-order model construction.

In the following, we show how to lift this limitation for symmetric
quadrature rules by incorporating derivative data into the
matrices~\cref{eqn:DataMatrices}.
Significantly, this will enable us to preserve state-space Hermiticity in the
reduced-order model, and thus asymptotic stability.


\subsection{Hermite data matrix formulae and algorithm}%
\label{sec:formulae}

The following theorem states the construction of the data matrices
$\EL$, $\AL$, $\BL$, and $\CL$ in~\cref{eqn:DataMatrices} in the case of
symmetric quadrature rules, providing explicit formulas for all four matrices.
In contrast to the formulas from~\cite{GosGB22}, \Cref{thm:btFromHData}
incorporates derivative information about the system's transfer function.

\begin{theorem} \label{thm:btFromHData}
  Assume that the left and right quadrature rules in~\cref{eqn:QuadRules}
  have the same number of terms $J = K = N$, and are symmetric with respect to
  the origin so that
  $\zeta_{i} = \omega_{i}$,
  $\vartheta_{i} = -\omega_{i}$, and
  $\varrho_{i} = \vartheta_{i} = \xi_{i}$, for all
  $i = 1, \ldots, N$.
  Furthermore, assume that the points $\i \omega_{i}$ and $-\i \omega_{i}$
  are not poles of the transfer function $\bG$, for $i = 1, \ldots, N$.
  Let the quadrature-based factors $\bRcheck\in\C^{n \times  mN}$ and
  $\bLcheck \in \C^{pN \times n}$ be given as
  in~\cref{eqn:QuadFactors}, and let the matrices
  $\EL, \AL \in \C^{pN \times mN}$,
  $\BL \in \C^{pN \times m}$, and
  $\CL \in \C^{p \times mN}$ be as in~\cref{eqn:DataMatrices}.
  Then, the matrices $\EL$, $\AL$, $\BL$ and $\CL$ are given
  block-entrywise by
  \begin{subequations} \label{eqn:dataFormulas_Hermite}
  \begin{align}
    \label{eqn:EfromData_Hermite}
    \big( \EL \big)_{\bk, \bj} & =
      \begin{cases}
        \displaystyle -\xi_{k} \xi_{j} \bG'(\i \omega_{j}) &
          \text{if} \quad \omega_{j} = -\omega_{k}, \\[10pt]
        \displaystyle -\xi_{k} \xi_{j}
          \frac{\bG(-\i \omega_{k}) - \bG(\i \omega_{j})}%
          {-\i \omega_{k} - \i \omega_{j}} & \text{if} \quad
          \omega_{j} \neq -\omega_{k},
      \end{cases}\\[2ex]
    \label{eqn:AfromData_Hermite}
    \big( \AL \big)_{\bk, \bj} & =
      \begin{cases}
        \displaystyle -\xi_{k} \xi_{j}
          \Big( \bG(\i \omega_{j}) + \i \omega_{j}
            \bG'(\i \omega_{j}) \Big)  &
          \text{if} \quad \omega_{j} = -\omega_{k}, \\[10pt]
    \displaystyle -\xi_{k} \xi_{j}
      \frac{-\i \omega_{k} \bG(-\i \omega_{k}) - \i \omega_{j}
        \bG(\i \omega_{j})}%
        {-\i \omega_{k} - \i \omega_{j}} & \text{if} \quad
        \omega_{j} \neq -\omega_{k},
      \end{cases}\\[2ex]
    \label{eqn:BfromData_Hermite}
    \big( \BL \big)_{\bk, :} & = \xi_{k} \bG(-\i \omega_{k}),\\[2ex]
    \label{eqn:CfromData_Hermite}
    \big( \CL \big)_{:, \bj} & = \xi_{j} \bG(\i \omega_{j}),
  \end{align}
  \end{subequations}
  for $k, j = 1,\ldots, N$, where the boldface $\bk$ and $\bj$
  denote $(p \times m)$-sized matrix blocks.
\end{theorem}
\begin{proof}
  The expressions in~\cref{eqn:dataFormulas_Hermite} for the case of 
  $\omega_{j} \neq -\omega_{k}$ directly follow from the original QuadBT
  framework in~\cite{GosGB22} with the choice of quadrature nodes
  $\zeta_{j} = \omega_{j}$ and $\vartheta_{k} = -\omega_{k}$
  in~\cref{eqn:QuadFactors}.
  It remains to be shown that the formulae~\cref{eqn:EfromData_Hermite}
  and~\cref{eqn:AfromData_Hermite} hold in the case of
  $\omega_{j} = -\omega_{k}$.
  To this end, we observe that by definition of the quadrature
  factors~\cref{eqn:QuadFactors} with $\omega_{j} = -\omega_{k}$
  and $\xi_{j} = \xi_{k}$, it holds that
  \begin{equation*}
    \big( \EL \big)_{\bk, \bj} =
      \bI_{k, p}^{\trans} \left( \bLcheck^{\herm} \bE \bRcheck \right)
      \bI_{j, m} =
      \xi_{j}^{2} \bC \bPhi(\i \omega_{j})^{-1} \bE
      \bPhi(\i \omega_{j})^{-1} \bB =
      -\xi_{j}^{2} \bG'(\i \omega_{j}),
  \end{equation*}
  which proves~\cref{eqn:EfromData_Hermite}.
  Similarly, the formula in~\cref{eqn:AfromData_Hermite} for
  $\omega_{j} = -\omega_{k}$ with $\xi_{j} = \xi_{k}$ follows,
  proving the theorem.
\end{proof}

We note that the results in \Cref{thm:btFromHData} are the first-order
formulation of the results that we developed for second-order dynamical systems
in~\cite[Thm.~2]{ReiW25a}.

\Cref{thm:btFromHData} provides the foundation for a QuadBT method that
incorporates symmetric quadrature rules and derivative data.
The term symmetric quadrature is chosen here because the assumptions made in
\Cref{thm:btFromHData} imply that if the node-weight pair
$(\omega_{j}, \xi_{j})$ is part of a chosen quadrature rule, then
$(-\omega_{j}, \xi_{j})$ must also be a pair in the quadrature rule.
The resulting computational procedure for the proposed symmetric Hermite
QuadBT method is summarized in \Cref{alg:HQuadBT}.
Note that the reduction order $r$ in \Cref{alg:HQuadBT} is assumed to
be at most the number of nonzero singular values of the matrix $\EL$ so that
$\bSigmacheck_{1}$ is invertible.

\begin{algorithm}[t]
  \SetAlgoHangIndent{1pt}
  \DontPrintSemicolon
  \caption{Symmetric Hermite quadrature-based balanced truncation.} 
  \label{alg:HQuadBT}

  \KwIn{System transfer function~$\bG$ with derivative $\bG'$,
    symmetric quadrature nodes and weights
    $\{(\omega_{i}, \xi_{i})\}_{i = 1}^{N}$ following
    \Cref{thm:btFromHData},
    reduction order $r$.}
  \vspace{2pt}
  \KwOut{Reduced-order system matrices $\bAr, \bBr, \bCr, \bEr$
    in~\cref{eqn:rom}.}

  Evaluate the transfer function $\bG$ at the quadrature nodes to obtain
    the data
    \begin{equation*}
      \left\{ \big( \bG(\i \omega_{i}), \bG'(\i \omega_{i}) \big)
        \right\}_{i = 1}^{N}
    \end{equation*}
    and construct the data matrices $\EL, \AL, \BL, \CL$
    according to~\cref{eqn:dataFormulas_Hermite}.\;
  
  Compute the singular value decomposition
    \begin{equation*}
      \EL = \begin{bmatrix} \bUcheck_{1} & \bUcheck_{2}\end{bmatrix}
        \begin{bmatrix} \bSigmacheck_{1} & \bzero \\ \bzero & \bSigmacheck_{2}
        \end{bmatrix}
        \begin{bmatrix} \bVcheck_{1}^{\herm} \\[2pt] \bVcheck_{2}^{\herm}
        \end{bmatrix},
    \end{equation*}
    for $\bSigmacheck_{1} \in \R^{r \times r}$ diagonal containing the
    $r$ largest nonzero singular values, and $\bUcheck_{1}$, $\bVcheck_{1}$,
    $\bSigmacheck_{2}$, $\bUcheck_{2}$, and $\bVcheck_{2}$ partitioned
    accordingly.\;
    
  Compute the reduced-order model matrices according to
    \begin{equation*}
      \bEr = \bI_{r}, \quad
      \bAr = \bSigmacheck_{1}^{-1/2} \bUcheck_{1}^{\herm} \AL
        \bVcheck_{1} \bSigmacheck_{1}^{-1/2}, \quad
      \bBr = \bSigmacheck_{1}^{-1/2} \bUcheck_{1}^{\herm} \BL, \quad
      \bCr = \CL \bVcheck_{1} \bSigmacheck_{1}^{-1/2}.
    \end{equation*}\;
    \vspace{-\baselineskip}
\end{algorithm}


\subsection{Stability preservation for state-space Hermitian, dissipative
  systems}%
\label{sec:Hermitian}

While \Cref{alg:HQuadBT} can be used for the construction of arbitrary
dynamical systems from given data~\cref{eqn:Hdata}, the use of symmetric
quadrature rules enables the preservation of Hermiticity in system matrices.
We say that a linear continuous-time model of the form~\cref{eqn:sys} is
state-space Hermitian, if the system matrices in~\cref{eqn:sys} satisfy
\begin{equation} \label{eqn:SSHsys}
  \bE = \bE^{\herm}, \quad
  \bA = \bA^{\herm}, \quad
  \bB = \bC^{\herm}.
\end{equation}
If additionally $\bE$ is symmetric positive definite and
$\bA$ is symmetric negative definite, then the state-space model is called
dissipative.
In the subsequent results, we show that
if there exists a state-space Hermitian, dissipative state-space
model~\cref{eqn:sys} of the system from which the
data~\cref{eqn:Hdata} are generated, then \Cref{alg:HQuadBT} will
produce a reduced-order model that preserves both of these qualities.

\begin{theorem} \label{thm:SSHsys}
  Let $\bE, \bA, \bB, \bC$ be a realization of the model corresponding to the
  transfer function $\bG$.
  If there exists a generalized state-space transformation with basis matrices
  $\bT, \bW \in \C^{n \times n}$ so that the model given by
  \begin{equation*}
    \tbE = \bW^{\herm} \bE \bT, \quad
    \tbA = \bW^{\herm} \bA \bT, \quad
    \tbB = \bW^{\herm} \bB, \quad
    \tbC = \bC \bT
  \end{equation*}
  is state-space Hermitian and dissipative, i.e., it holds
  \begin{equation} \label{eqn:SSHDiss}
    \tbE = \tbE^{\herm} > 0, \quad
    \tbA = \tbA^{\herm} < 0, \quad
    \tbB = \tbC^{\herm},
  \end{equation}
  then for any quadrature rule that satisfies the assumptions of
  \Cref{thm:btFromHData}, it holds that the data matrices constructed
  via~\cref{eqn:dataFormulas_Hermite} satisfy
  \begin{equation} \label{eqn:SSHLoewner}
    \EL = \EL^{\herm} \geq 0, \quad
    \AL = \AL^{\herm} \leq 0, \quad
    \BL = \CL^{\herm}.
  \end{equation}
  Moreover, reduced-order models with the matrices $\bEr, \bAr, \bBr, \bCr$
  constructed via \Cref{alg:HQuadBT} are state-space Hermitian and
  dissipative for any reduction order $r$ for which $\bSigmacheck_{1}$
  is invertible.
\end{theorem}
\begin{proof}
  First, we note that the data matrices in~\cref{eqn:dataFormulas_Hermite}
  are realization independent as they are constructed only via transfer function
  data and the transfer functions of the model described by
  $\bE, \bA, \bB, \bC$ and $\tbE, \tbA, \tbB, \tbC$ are identical;
  see~\cite{Ant05}.
  Thus, for simplicity of notation, we assume that $\bE, \bA, \bB, \bC$ is the
  state-space Hermitian, dissipative system realization
  satisfying~\cref{eqn:SSHDiss}.
  Following the assumption imposed on the quadrature rules in
  \Cref{thm:btFromHData} and the fact that the system realization is
  state-space Hermitian, the quadrature-based factors
  in~\cref{eqn:QuadFactors} satisfy $\bRcheck = \bLcheck$.
  Using \Cref{thm:btFromHData} and again the fact that the system
  corresponding to $\bG$ is state-space Hermitian, it holds that
  \begin{subequations}
  \begin{alignat*}{3}
    \EL & = \bRcheck^{\herm} \bE \bRcheck &&
      = \bRcheck^{\herm} \bE^{\herm} \bRcheck &&
      = \EL^{\herm}, \\
    \AL & = \bRcheck^{\herm} \bA \bRcheck &&
      = \bRcheck^{\herm} \bA^{\herm} \bRcheck &&
      = \AL^{\herm}, \quad \text{and} \\
    \BL & = \bRcheck^{\herm} \bB &&
      = \bRcheck^{\herm} \bC^{\herm} &&
      = \CL^{\herm}.
  \end{alignat*}
  \end{subequations}
  Furthermore, for any $\bz \in \C^{mN}$, it holds that
  \begin{equation*}
    \bz^{\herm} \EL \bz =
      \underbrace{\bz^{\herm} \bRcheck^{\herm}}_{=\tbz^{\herm}} \bE
      \underbrace{\bRcheck \bz}_{=\tbz}
      = \tbz^{\herm} \bE \tbz \geq 0
    \quad\text{and}\quad
    \bz^{\herm} \AL \bz =
      \underbrace{\bz^{\herm} \bRcheck^{\herm}}_{=\tbz^{\herm}} \bA
      \underbrace{\bRcheck \bz}_{=\tbz}
      = \tbz^{\herm} \bA \tbz \leq 0,
  \end{equation*}
  which proves~\cref{eqn:SSHLoewner}.
  Next, we prove the claims on the constructed reduced-order models.
  Since $\EL$ is Hermitian positive semi-definite, for its singular value
  decomposition, it holds that
  \begin{equation*}
    \EL =  \begin{bmatrix} \bUcheck_{1} & \bUcheck_{2}\end{bmatrix}
      \begin{bmatrix} \bSigmacheck_{1} & \bzero \\ \bzero & \bSigmacheck_{2}
      \end{bmatrix}
      \begin{bmatrix} \bUcheck_{1}^{\herm} \\[2pt] \bUcheck_{2}^{\herm}
      \end{bmatrix}.
  \end{equation*}
  Following the construction of the reduced-order matrices in
  \Cref{alg:HQuadBT}, we see that
  \begin{subequations}
  \begin{align*}
    \bEr & = \bI_{r} = \bEr^{\herm} > 0, \\
    \bAr & = \bSigmacheck_{1}^{-1/2} \bUcheck_{1}^{\herm} \AL
      \bUcheck_{1} \bSigmacheck_{1}^{-1/2} =
      \bAr^{\herm}, \quad\text{and}\\
    \bBr & = \bSigmacheck_{1}^{-1/2} \bUcheck_{1}^{\herm} \BL =
      (\CL \bUcheck_{1} \bSigmacheck_{1}^{-1/2})^{\herm} =
      \bCr^{\herm}
  \end{align*}
  \end{subequations}
  hold.
  It is left to show that $\bAr$ is negative definite.
  To this end, we observe that by construction of $\bUcheck_{1}$ and
  $\bSigmacheck_{1}$, the matrix product
  $\bRcheck \bUcheck_{1} \bSigmacheck_{1}^{-1/2}$ has only a trivial
  kernel so that for all $\bzero \neq \bz \in \C^{r}$, we have that
  $\tbz = \bRcheck \bUcheck_{1} \bSigmacheck_{1}^{-1/2} \bz \neq \bzero$.
  It follows that
  \begin{equation*}
    \bz^{\herm} \bAr \bz = \tbz^{\herm} \bA \tbz < 0
    \quad\text{for all}\quad
    \bzero \neq \bz \in \C^{r}.
  \end{equation*}
  This shows that the reduced-order model is state-space Hermitian and
  dissipative for any $r$ for which $\bSigmacheck_{1}$ is invertible,
  thus concluding the proof.
\end{proof}

We emphasize that the results of \Cref{thm:SSHsys} are independent of
\emph{any} explicit realization of the system~\cref{eqn:sys} underlying the
used data.
The theorem states that \Cref{alg:HQuadBT} allows the construction of
system realizations with desired properties solely from data.
It also allows for the data-driven verification of system properties by
checking the properties of the data matrices in~\cref{eqn:SSHLoewner}.

A direct consequence of state-space Hermiticity and dissipativity is asymptotic
stability of the linear system, i.e., the property that the
the eigenvalues of the matrix pencil $\lambda \bE - \bA$ lie in the
open left half-plane.
Consequently, it follows from \Cref{thm:SSHsys} that
\Cref{alg:HQuadBT} preserves asymptotic stability for such systems.

\begin{corollary} \label{cor:stable}
  Let the assumptions of \Cref{thm:SSHsys} hold.
  Reduced-order models constructed via \Cref{alg:HQuadBT} are
  asymptotically stable for all reduction orders $r$ for which
  $\bSigmacheck_{1}$ is invertible.
\end{corollary}
\begin{proof}
  This result follows directly from \Cref{thm:SSHsys} since the
  eigenvalues of any matrix pencil $\lambda \bEr - \bAr$ with
  $\bEr = \bEr^{\herm} > 0$ and $\bAr = \bAr^{\herm} < 0$ must lie in the
  open left half-plane~\cite{GolV13}.
\end{proof}

Both \Cref{thm:SSHsys,cor:stable} hold for arbitrary
symmetric quadrature rules that satisfy the assumptions of
\Cref{thm:btFromHData}.
This implies that the preservation of state-space Hermiticity and dissipativity,
and consequently asymptotic stability, can be guaranteed independent of the
accuracy of the quadrature rule.


\subsection{Real system matrices and state-space symmetry}%
\label{sec:real}

In many applications, the matrices describing the dynamical
system~\cref{eqn:sys} are real, and it is desired to preserve this property in
the reduced-order system~\cref{eqn:rom} learned from data.
In the case of real system matrices in~\cref{eqn:sys}, the corresponding
transfer function~\cref{eqn:tf} commutes with conjugation so that
\begin{equation} \label{eqn:conjugate}
  \overline{\bG(s)} = \bG(\overline{s})
  \quad\text{for all}\quad s \in \C.
\end{equation}
The same property~\cref{eqn:conjugate} must hold for the
data~\cref{eqn:Hdata} to correspond to a real system, meaning if
$(\i \omega, \bG(\i \omega), \bG'(\i \omega))$ is a data point of
the underlying system, so is  $(-\i \omega, \overline{\bG(\i \omega)},
\overline{\bG'(\i \omega)})$.
Let the symmetric quadrature formula in \Cref{thm:btFromHData} be ordered
in such a way that
\begin{equation}
\label{eqn:symmNodes}
  \omega_{1} = -\omega_{2}, \quad
  \omega_{3} = -\omega_{4}, \quad
  \ldots, \quad
  \omega_{N - 2} = -\omega_{N - 1}, \quad
  \omega_{N} = 0,
\end{equation}
let $N > 1$ be an odd number, and let the transfer function $\bG$ from which
the data is generated satisfy~\cref{eqn:conjugate}.
Then, there exist the following two transformation matrices
\begin{subequations} \label{eqn:realTrafo}
\begin{align} 
  \bJr & = \begin{bmatrix} \bI_{(N - 1) / 2} \otimes
    \frac{1}{\sqrt{2}} \begin{bmatrix} \bI_{m} & -\i \bI_{m} \\ \bI_{m} &
    \hphantom{-}\i\bI_{m} \end{bmatrix} & \bzero \\ \bzero & \bI_{m}
    \end{bmatrix}
  \quad\text{and}\\
  \bJl & = \begin{bmatrix} \bI_{(N - 1) / 2} \otimes
    \frac{1}{\sqrt{2}} \begin{bmatrix} \bI_{p} & -\i \bI_{p} \\ \bI_{p} &
    \hphantom{-}\i\bI_{p} \end{bmatrix} & \bzero \\ \bzero & \bI_{p}
    \end{bmatrix},
\end{align}
\end{subequations}
where $\otimes$ denotes the Kronecker product, so that the transformed data
matrices $\bJl^{\herm} \EL \bJr$, $\bJl^{\herm} \AL \bJr$,
$\bJl^{\herm} \BL$, and $\CL \bJr$ are real.
Adding this transformation in between the Lines~1 and~2 in
\Cref{alg:HQuadBT} enforces the construction of models with real
matrices.
We note that if $\omega = 0$ is not a quadrature node, then the last block rows
and columns of $\bJr$ and $\bJl$ can be removed;
see, for example,~\cite{AntBG20, ReiW25a}.

In the case of real system matrices, state-space Hermiticity~\cref{eqn:SSHsys}
simplifies to state-space symmetry so that
\begin{equation*}
\bE = \bE^{\trans}, \quad
\bA = \bA^{\trans}, \quad
\bB = \bC^{\trans}.
\end{equation*}
In this case, the transformation matrices in~\cref{eqn:realTrafo} become
equal, $\bJ = \bJr = \bJl$, and the results of \Cref{thm:SSHsys} and
\Cref{cor:stable} hold for the transformed data matrices
$\bJ^{\herm} \EL \bJ$, $\bJ^{\herm} \AL \bJ$, $\bJ^{\herm} \BL$, and $\CL \bJ$,
and the real reduced-order matrices constructed by \Cref{alg:HQuadBT}.
In other words, \Cref{alg:HQuadBT} does preserve state-space symmetry
and dissipativity for real systems.


\subsection{Systems with feed-through terms}%
\label{sec:Dterm}

All results presented is this work also directly apply to systems with
feed-through terms $\bD \in \C^{p \times m}$ of the form
\begin{equation*}
  \bE \dot{\bx}(t) = \bA \bx(t) + \bB \bu(t), \quad
  \by(t) = \bC \bx(t) + \bD \bu(t).
\end{equation*}
To this end, the modified transfer function $\tbG(s) = \bG(s) - \bD$ has to be
used in all formulas above that involve the system's transfer function.
For example, the data used takes the form
\begin{equation*}
  \left\{ \big( \i \omega_{i}, \tbG(\i \omega_{i}), \tbG'(\i \omega_{i}) \big)
    \right\}_{i = 1}^{N}.
\end{equation*}
We note that $\tbG' = \bG'$ holds for the first derivative of the system's and
modified transfer function.
For the reduced-order system, we then have that
\begin{equation*}
  \bEr \dot{\bxr}(t) = \bAr \bxr(t) + \bBr \bu(t), \quad
  \byr(t) = \bCr \bxr(t) + \bD \bu(t),
\end{equation*}
where the matrices $\bEr, \bAr, \bBr, \bCr$ are constructed via
\Cref{alg:HQuadBT} using $\tbG$ instead of $\bG$.
If unknown, the feed-through term can be determined via sampling at high
frequency since $\lim_{s \to \infty} \bG(s) = \bD$.


\section{Numerical experiments}%
\label{sec:numerics}

We demonstrate the proposed symmetric Hermite formulation of QuadBT on two
dynamical system benchmarks from the model reduction literature.
The reported numerical experiments have been performed on a MacBook Air with
8\,GB of RAM and an Apple M2 processor running macOS Ventura version 13.4 with
MATLAB 23.2.0.2515942 (R2023b) Update 7.
The source codes, data, and results of the numerical experiments are available
at~\cite{supReiW26}.


\subsection{Experimental setup}%
\label{sec:setup}

In the subsequent numerical experiments, we compute data-based reduced
models~\cref{eqn:rom} using the following approaches:
\begin{description}
  \item[\symQuadbt{}] is the proposed symmetric Hermite formulation of QuadBT
    in \Cref{alg:HQuadBT};
  \item[\Quadbt{}] is the (non-symmetric) data-driven balanced truncation
    from~\cite{GosGB22};
  \item[\FDsymQuadbt{}] is a variation of \symQuadbt{} where central finite
    differences with a relative step-size of $h\approx 10^{-8}\omega_{k}$
    are used to approximate the values of the transfer
    function~$\bG'(\i \omega_{k})$.
\end{description}
We include \FDsymQuadbt{} as a point of comparison to investigate how the
proposed Hermite-based \symQuadbt{} performs when exact derivative evaluations
of the transfer function are unavailable.

For all of the quadrature-based reduced models, we employ an exponential
trapezoidal quadrature rule to implicitly approximate the Gramians
in~\cref{eqn:Gramians}.
For simplicity, an equal number of left and right quadrature nodes is used,
i.e., $J = K = N$.
In each example, the quadrature rules are applied in chosen intervals along
the imaginary axis; we report these intervals in the subsequent sections.
For the symmetric methods \symQuadbt{} and \FDsymQuadbt{}, we choose the left
and right quadrature nodes to be equal as in \Cref{thm:btFromHData} so
that we have $\vartheta_{i} = -\omega_{i}$, $\zeta_{i} = \omega_{i}$, and 
$\xi_{i} = \varphi_{i} = \varrho_{i}$, for all $i = 1, 2, \ldots, N$.
Moreover, we assume that the nodes are distributed symmetrically along the
imaginary axis according to~\cref{eqn:symmNodes}, without the final
$\omega_{N} = 0$ since $N$ is chosen to be even for our experiments. 
For the non-symmetric \Quadbt{}, we assume that the left and right quadrature
nodes are distributed symmetrically along the imaginary axis so that
\begin{align*}
  \vartheta_{1} & < \vartheta_{3} < \ldots < \vartheta_{N - 1} < 0 <
    \vartheta_{2} < \vartheta_{4} < \ldots < \vartheta_{N}, \\
  \zeta_{1} & < \zeta_{3} < \ldots < \zeta_{N-1} < 0 <
    \zeta_{2} < \zeta_{4} < \ldots < \zeta_{N}.
\end{align*}
This enables a real-valued construction of the \Quadbt{} reduced models by
following a similar procedure to what is described in
\Cref{sec:real}; see~\cite[Sec.~4.1]{GosGB22} for further details.
For each of the employed data-driven modeling approaches, we enforce such a
real-valued construction.

For the presentation of the results, we use the following performance
measures.
For visual comparisons, we plot the magnitude of full- and reduced-order
transfer functions at discrete points on the positive imaginary axis.
We also show the pointwise relative approximation errors of the transfer
functions
\begin{equation} \label{eqn:relerr}
  \relerr(\i z_{k}) = \frac{\lVert \bG(\i z_{k}) -
    \bGr(\i z_{k}) \rVert_{2}}{\lVert \bG(\i z_k) \rVert_{2}},
\end{equation}
at discrete frequencies $z_{k} \in \CI$ from the finite interval
$\CI = [\zMin, \zMax] \subset [0, \infty)$.
The specific choice of $\CI$ varies between the examples.
Additionally, we score the performance of the different methods by using local
approximations to the relative $\CH_{\infty}$ error and relative $\CH_{2}$
error via
\begin{subequations} \label{eqn:Hrelerrs}
\begin{align} 
  \relerr_{\CH_{\infty}} & = \frac{\max_{z_{k} \in \CI} \lVert
    \bG(\i z_{k}) - \bGr(\i z_{k}) \rVert_{2}}%
    {\max_{z_{k} \in \CI} \lVert \bG(\i z_{k}) \rVert_{2}}
    \quad\text{and}\\
  \relerr_{\CH_{2}} & = \sqrt{\frac{\sum\limits_{z_{k} \in \CI} \lVert
    \bG(\i z_{k}) - \bGr(\i z_{k}) \rVert_{\frob}^2}%
    {\sum\limits_{z_{k} \in \CI} \lVert \bG(\i z_{k})
    \rVert_{\frob}^2}}.
\end{align}
\end{subequations}


\subsection{Butterfly gyroscope}%
\label{sec:butterfly}

The first example that we consider is the butterfly gyroscope from
the Oberwolfach Benchmark Collection~\cite{Bil05, morwiki_gyro}, which models a
vibrating mechanical structure used for inertial navigation.
The model is expressed as a first-order system~\cref{eqn:sys} with
$n = 34\,722$ states, $m = 1$ input, and $p = 12$ outputs.
The outputs measure the displacement of the four detection electrodes in the 
spatial directions.
In first-order form, the butterfly gyroscope is not state-space symmetric.
We investigate this benchmark as a proof of generalization for the proposed
symmetric Hermite formulation of QuadBT.

\begin{figure}[t]
  \centering
  \begin{subfigure}[b]{.49\linewidth}
    \centering
  \tikzexternalenable%
  \tikzsetnextfilename{Butterfly_r25_mag}%
  \begin{tikzpicture}[font = \plotfontsize]
  \pgfplotstableread{graphics/data/butterfly_r25_Response.dat}\tableINPUT
  
  \begin{loglogaxis}[%
    width  = .725\textwidth,
    height = .11\textheight,
    scale only axis,
    xmin = 1e4,
    xmax = 1e6,
    ymin = 1e-8,
    ymax = 1e-1,
    xminorticks = true,
    yminorticks = true,
    xlabel = {frequency $\omega$ (rad/s)},
    ylabel = {magnitude},
    ylabel style   = {yshift = -.3em},
    xlabel style   = {yshift = .3em},
    scaled x ticks = false,
    x tick label style = {/pgf/number format/1000 sep={\,}},
    y tick label style = {/pgf/number format/1000 sep={\,}},
    cycle list name    = plotlist
  ]
  
    \foreach \y in {1, 2, ..., 7}{
      \addplot+ table[x index = 0, y index = \y] {\tableINPUT};
    }
  \end{loglogaxis}
\end{tikzpicture}%
  \tikzexternaldisable%

    \caption{frequency response}
  \end{subfigure}%
  \hfill%
  \begin{subfigure}[b]{.49\linewidth}
    \centering
  \tikzexternalenable%
  \tikzsetnextfilename{Butterfly_r25_error}%
  \begin{tikzpicture}[font = \plotfontsize]
  \pgfplotstableread{graphics/data/butterfly_r25_Error.dat}\tableINPUT
  
  \begin{loglogaxis}[%
    width  = .725\textwidth,
    height = .11\textheight,
    scale only axis,
    xmin = 1e4,
    xmax = 1e6,
    ymin = 1e-6,
    ymax = 1e2,
    xminorticks = true,
    yminorticks = true,
    xlabel = {frequency $\omega$ (rad/s)},
    ylabel = {$\relerr(\i\omega)$},
    ylabel style   = {yshift = -.3em},
    xlabel style   = {yshift = .3em},
    scaled x ticks = false,
    x tick label style = {/pgf/number format/1000 sep={\,}},
    y tick label style = {/pgf/number format/1000 sep={\,}},
    cycle list name    = plotlist
  ]

  \pgfplotsset{cycle list shift = 1}
  
    \foreach \y in {1, 2, ..., 6}{
      \addplot+ table[x index = 0, y index = \y] {\tableINPUT};
    }
  \end{loglogaxis}
\end{tikzpicture}%
  \tikzexternaldisable%

    \caption{pointwise relative error}
  \end{subfigure}

  \vspace{.5\baselineskip}
  \resizebox{\linewidth}{!}{%
  \tikzexternalenable%
  \tikzsetnextfilename{legend}%
  \begin{tikzpicture}[font = \plotfontsize]
  \begin{axis}[%
    hide axis,
    width  = 1mm,
    height = 1mm,
    scale only axis,
    xmin = 0,
    xmax = 1,
    ymin = 0,
    ymax = 1,
    legend columns = 4, 
    legend style   = {
      at     = {(0,0)},
      anchor = center,
      /tikz/every even column/.append style = {column sep = 0.2cm}},
    legend cell align  = {left},
    clip mode          = individual,
    cycle list name    = plotlist]

    \addplot+ coordinates{ (0, 0) };
    \addlegendentry{Original model}

    \addplot+ coordinates{ (0, 0) };
    \addlegendentry{\Quadbt{} $(N=50)$}

    \addplot+ coordinates{ (0, 0) };
    \addlegendentry{\symQuadbt{} $(N=50)$}

    \addplot+ coordinates{ (0, 0) };
    \addlegendentry{\FDsymQuadbt{} $(N=50)$}

    \addlegendimage{/pgfplots/legend image code/.code={}}
    \addlegendentry{}

    \addplot+ coordinates{ (0, 0) };
    \addlegendentry{\Quadbt{} $(N=200)$}

    \addplot+ coordinates{ (0, 0) };
    \addlegendentry{\symQuadbt{} $(N=200)$}

    \addplot+ coordinates{ (0, 0) };
    \addlegendentry{\FDsymQuadbt{} $(N=200)$}
  \end{axis}
\end{tikzpicture}%
  \tikzexternaldisable%
}
  
  \caption{Frequency response and pointwise relative errors for $r=25$
    reduced-order models of the butterfly gyroscope benchmark.}
  \label{fig:ButterflyNumerics}
\end{figure}
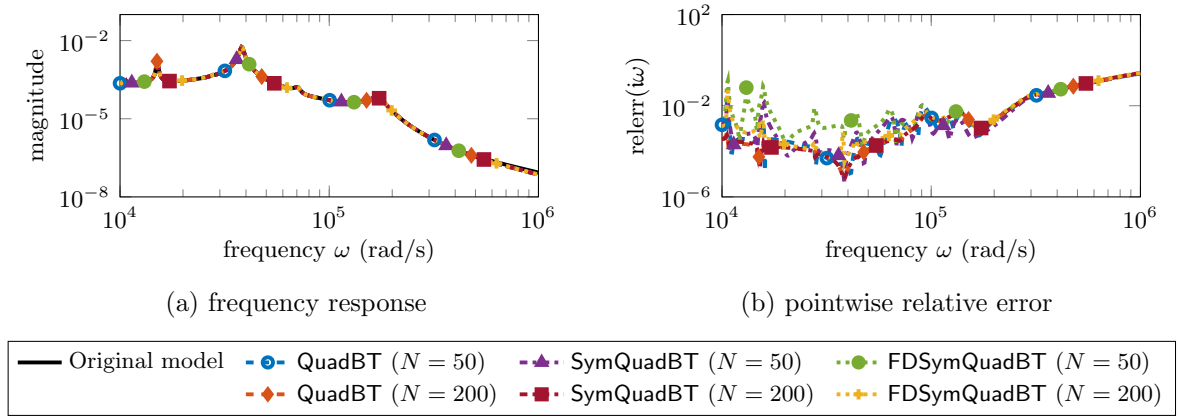

\begin{table}[t]
  \centering
  \caption{Relative $\CH_{\infty}$ errors and
    $\CH_2$ errors~\cref{eqn:Hrelerrs} for the $r = 25$ reduced-order models of
    the butterfly gyroscope for $N = 50, 200$ quadrature nodes.
    The smallest error is highlighted in \textbf{boldface}.}
  \label{tab:ButterflyRelErrors}
  \vspace{.5\baselineskip}

  \begin{tabular}{lrr}
    \hline\noalign{\smallskip} 
      & \multicolumn{1}{c}{$\relerr_{\CH_\infty}$}
      & \multicolumn{1}{c}{$\relerr_{\CH_2}$} \\
    \noalign{\smallskip}\hline\noalign{\smallskip}
    \Quadbt{} $(N = 50)$ &
      $1.0035\texttt{e-}3$ &
      $5.3610\texttt{e-}4$ \\
    \symQuadbt $(N = 50)$ & 
      $1.1759\texttt{e-}3$ &
      $8.1926\texttt{e-}4$ \\
    \FDsymQuadbt $(N = 50)$ &
      $6.2749\texttt{e-}3$ &
      $5.0262\texttt{e-}3$ \\
    \noalign{\smallskip}\hline\noalign{\smallskip}
    \Quadbt $(N = 200)$ &
      $\boldsymbol{4.0809\texttt{e-}5}$ &
      $\boldsymbol{1.0796\texttt{e-}4}$ \\
    \symQuadbt $(N = 200)$ &
      $4.0920\texttt{e-}5$ &
      $1.0810\texttt{e-}4$ \\
    \FDsymQuadbt $(N = 200)$ &
      $2.4033\texttt{e-}3$ &
      $1.0900\texttt{e-}3$ \\
    \noalign{\smallskip}\hline\noalign{\smallskip}
  \end{tabular}
\end{table}

Multiple reduced-order models of order $r = 25$ are computed using the
\Quadbt{}, \symQuadbt{}, and \FDsymQuadbt{} approaches.
For each approach, a reduced-order model is obtained by projecting the
appropriate data matrices generated from $N = 50$ and $N = 200$ quadrature
nodes, which are chosen to be logarithmically spaced points within
$-\i[10^{4}, 10^{6}] \cup \i[10^{4}, 10^{6}]$.
The frequency response and pointwise relative errors~\cref{eqn:relerr} of the
reduced-order models are presented in \Cref{fig:ButterflyNumerics}, and
the relative error measures~\cref{eqn:Hrelerrs} for the different methods are
shown in \Cref{tab:ButterflyRelErrors}. 
We observe that each of the \Quadbt{}, \symQuadbt{}, and \FDsymQuadbt{} models
provides very satisfactory approximations to the response of the full-order
transfer function.
For this example, as the number of quadrature nodes $N$ increases, the
corresponding reduced-order models become more accurate in the relative error
metrics~\cref{eqn:Hrelerrs}.
The finite difference-based reduced models perform marginally worse at the
lower frequencies, but still provide very satisfactory approximations,
suggesting that the symmetric Hermite formulation of QuadBT can be used even if
derivative data are not available.


\subsection{Steel profile}%
\label{sec:steel}

The second example that we consider is a semi-discretized heat transfer problem
used in the optimal cooling of steel profiles from the Oberwolfach
Benchmark Collection~\cite{BenS05b, morwiki_steel}.
The discretized model is a real-valued linear system with $n = 20\,209$ states
with symmetric $\bE$ and $\bA$ matrices.
To make the system state-space symmetric, we modify the outputs to be
collocated with the inputs so that $\bC = \bB^{\trans}$ with $m = p = 7$.
We investigate this benchmark to demonstrate the structure (symmetry)
preservation of the proposed \symQuadbt{} approach.

As in the previous example, multiple reduced-order models of order $r = 25$
are computed using the \Quadbt{}, \symQuadbt{}, and \FDsymQuadbt{} approaches. 
For each approach, a reduced-order model is obtained by projecting the
appropriate data matrices generated from $N = 50$ and $N = 200$ quadrature
nodes, which are chosen to be logarithmically equidistant points within
$-\i[10^{-6}, 10^{2}] \cup \i[10^{-6}, 10^{2}]$.
The frequency response and pointwise relative errors~\cref{eqn:relerr} of the
reduced-order models are presented in \Cref{fig:RailNumerics}, and the
relative error measures~\cref{eqn:Hrelerrs} for the different methods are
presented in \Cref{tab:RailRelErrors}. 
All of the reduced-order models provide satisfactory reconstructions of
the error.
For this example, the $N = 50$ reduced models perform marginally better than
their $N = 200$ counterparts.

We validate the symmetry and stability preservation results of
\Cref{thm:SSHsys,cor:stable} in
\Cref{tab:RailSymmErrors,tab:stabCheck}. 
For the stability result, a hierarchy of reduced models of orders
$r = 5, 10, \ldots, 25$ are computed via \Quadbt{}, \symQuadbt{},
and \FDsymQuadbt{} using $N = 50$ quadrature nodes. 
By construction, both \symQuadbt{} and \FDsymQuadbt{} preserve state-space
symmetry and asymptotic stability.
On the other hand, the \Quadbt{} reduced models happen to preserve stability as
well, but the underlying state-space symmetry is destroyed in the data-driven
construction.
The \FDsymQuadbt{} reduced models are also asymptotically stable, and nearly
symmetric as well, although this is not preserved to machine precision due
to round-off errors. 

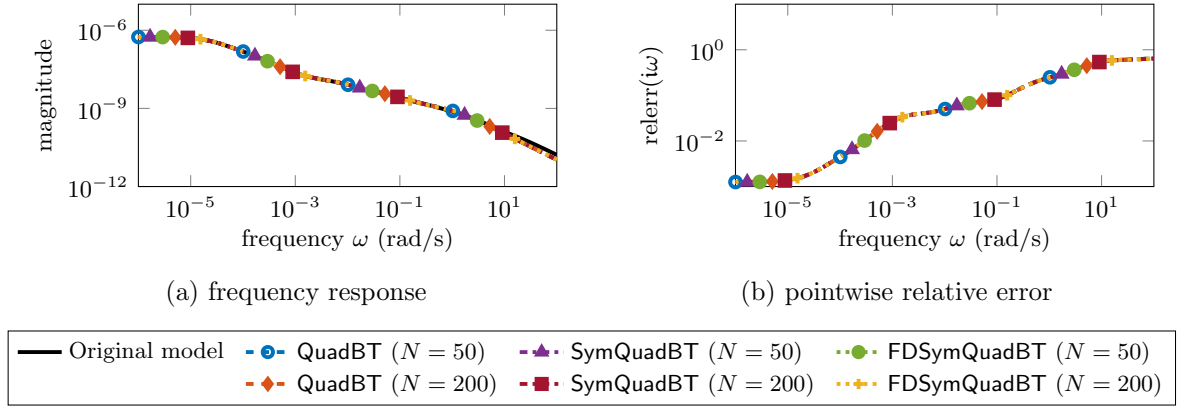
\begin{figure}[t]
  \centering
  \begin{subfigure}[b]{.49\linewidth}
    \centering
  \tikzexternalenable%
  \tikzsetnextfilename{Rail_r25_mag}%
  \begin{tikzpicture}[font = \plotfontsize]
  \pgfplotstableread{graphics/data/rail_r25_Response.dat}\tableINPUT
  
  \begin{loglogaxis}[%
    width  = .725\textwidth,
    height = .11\textheight,
    scale only axis,
    xmin = 1e-6,
    xmax = 1e2,
    ymin = 1e-12,
    ymax = 1e-5,
    xminorticks = false,
    yminorticks = true,
    xlabel = {frequency $\omega$ (rad/s)},
    ylabel = {magnitude},
    ylabel style   = {yshift = -.3em},
    xlabel style   = {yshift = .3em},
    scaled x ticks = false,
    x tick label style = {/pgf/number format/1000 sep={\,}},
    y tick label style = {/pgf/number format/1000 sep={\,}},
    cycle list name    = plotlist
  ]
  
    \foreach \y in {1, 2, ..., 7}{
      \addplot+ table[x index = 0, y index = \y] {\tableINPUT};
    }
  \end{loglogaxis}
\end{tikzpicture}%
  \tikzexternaldisable%

    \caption{frequency response}
  \end{subfigure}%
  \hfill%
  \begin{subfigure}[b]{.49\linewidth}
    \centering
  \tikzexternalenable%
  \tikzsetnextfilename{Rail_r25_error}%
  \begin{tikzpicture}[font = \plotfontsize]
  \pgfplotstableread{graphics/data/rail_r25_Error.dat}\tableINPUT
  
  \begin{loglogaxis}[%
    width  = .725\textwidth,
    height = .11\textheight,
    scale only axis,
    xmin = 1e-6,
    xmax = 1e2,
    ymin = 1e-3,
    ymax = 1e1,
    xminorticks = false,
    yminorticks = true,
    xlabel = {frequency $\omega$ (rad/s)},
    ylabel = {$\relerr(\i\omega)$},
    ylabel style   = {yshift = -.3em},
    xlabel style   = {yshift = .3em},
    scaled x ticks = false,
    x tick label style = {/pgf/number format/1000 sep={\,}},
    y tick label style = {/pgf/number format/1000 sep={\,}},
    cycle list name    = plotlist
  ]

  \pgfplotsset{cycle list shift = 1}
  
    \foreach \y in {1, 2, ..., 6}{
      \addplot+ table[x index = 0, y index = \y] {\tableINPUT};
    }
  \end{loglogaxis}
\end{tikzpicture}%
  \tikzexternaldisable%

    \caption{pointwise relative error}
  \end{subfigure}

  \vspace{.5\baselineskip}
  \resizebox{\linewidth}{!}{%
  \tikzexternalenable%
  \tikzsetnextfilename{legend}%
  \begin{tikzpicture}[font = \plotfontsize]
  \begin{axis}[%
    hide axis,
    width  = 1mm,
    height = 1mm,
    scale only axis,
    xmin = 0,
    xmax = 1,
    ymin = 0,
    ymax = 1,
    legend columns = 4, 
    legend style   = {
      at     = {(0,0)},
      anchor = center,
      /tikz/every even column/.append style = {column sep = 0.2cm}},
    legend cell align  = {left},
    clip mode          = individual,
    cycle list name    = plotlist]

    \addplot+ coordinates{ (0, 0) };
    \addlegendentry{Original model}

    \addplot+ coordinates{ (0, 0) };
    \addlegendentry{\Quadbt{} $(N=50)$}

    \addplot+ coordinates{ (0, 0) };
    \addlegendentry{\symQuadbt{} $(N=50)$}

    \addplot+ coordinates{ (0, 0) };
    \addlegendentry{\FDsymQuadbt{} $(N=50)$}

    \addlegendimage{/pgfplots/legend image code/.code={}}
    \addlegendentry{}

    \addplot+ coordinates{ (0, 0) };
    \addlegendentry{\Quadbt{} $(N=200)$}

    \addplot+ coordinates{ (0, 0) };
    \addlegendentry{\symQuadbt{} $(N=200)$}

    \addplot+ coordinates{ (0, 0) };
    \addlegendentry{\FDsymQuadbt{} $(N=200)$}
  \end{axis}
\end{tikzpicture}%
  \tikzexternaldisable%
}
  
  \caption{Frequency response and pointwise relative errors for $r=25$
    reduced-order models of the steel profile benchmark.}
  \label{fig:RailNumerics}
\end{figure}

\begin{table}[t]
  \centering
  \caption{Relative $\CH_{\infty}$ errors and
    $\CH_2$ errors~\cref{eqn:Hrelerrs} for the $r = 25$ reduced-order models of
    the steel profile for $N = 50, 200$ quadrature nodes.
    The smallest error is highlighted in \textbf{boldface}.}
  \label{tab:RailRelErrors}
  \vspace{.5\baselineskip}

  \begin{tabular}{lrr}
    \hline\noalign{\smallskip} 
      & \multicolumn{1}{c}{$\relerr_{\CH_\infty}$}
      & \multicolumn{1}{c}{$\relerr_{\CH_2}$} \\
    \noalign{\smallskip}\hline\noalign{\smallskip}
    \Quadbt $(N = 50)$ &
      $7.9881\texttt{e-}2$ &
      $6.3279\texttt{e-}2$ \\
    \symQuadbt $(N = 50)$ &
      $\boldsymbol{6.6404\texttt{e-}2}$ &
      $\boldsymbol{5.2536\texttt{e-}2}$ \\
    \FDsymQuadbt $(N = 50)$ &
      $\boldsymbol{6.6404\texttt{e-}2}$ &
      $\boldsymbol{5.2536\texttt{e-}2}$ \\
    \noalign{\smallskip}\hline\noalign{\smallskip}
    \Quadbt $(N = 200)$ &
      $7.8907\texttt{e-}2$ &
      $6.2459\texttt{e-}2$ \\
    \symQuadbt $(N = 200)$ &
      $7.5467\texttt{e-}2$ &
      $5.9711\texttt{e-}2$ \\
    \FDsymQuadbt $(N = 200)$ &
      $7.5467\texttt{e-}2$ &
      $5.9711\texttt{e-}2$ \\
    \noalign{\smallskip}\hline\noalign{\smallskip}
  \end{tabular}
\end{table}

\begin{table}[t]
  \centering
  \caption{State-space symmetry errors for the $r = 25$ reduced-order models of
    the steel profile for $N = 50, 200$ quadrature nodes.}
  \label{tab:RailSymmErrors}
  \vspace{.5\baselineskip}

  \begin{tabular}{lrr}
    \hline\noalign{\smallskip} 
      & \multicolumn{1}{c}{$\|\bAr-\bAr^{\trans}\|_{\frob}/\|\bAr\|_{\frob}$}
      & \multicolumn{1}{c}{$\|\bBr-\bCr^{\trans}\|_{\frob}/\|\bBr\|_{\frob}$} \\
    \noalign{\smallskip}\hline\noalign{\smallskip}
    \Quadbt $(N = 50)$ &
      $3.1603\texttt{e-}3$ &
      $2.4869\texttt{e-}2$ \\
    \symQuadbt $(N = 50)$ &
      $3.2025\texttt{e-}15$ &
      $3.0032\texttt{e-}15$ \\
    \FDsymQuadbt $(N = 50)$ &
      $3.3799\texttt{e-}9$ &
      $7.3136\texttt{e-}9$ \\
    \noalign{\smallskip}\hline\noalign{\smallskip}
    \Quadbt $(N = 200)$ &
      $5.9561\texttt{e-}4$ &
      $5.5025\texttt{e-}3$ \\
    \symQuadbt $(N = 200)$ &
      $4.9731\texttt{e-}15$ &
      $4.4158\texttt{e-}15$ \\
    \FDsymQuadbt $(N = 200)$ &
      $1.6389\texttt{e-}10$ &
      $4.9523\texttt{e-}10$ \\
    \noalign{\smallskip}\hline\noalign{\smallskip}
  \end{tabular}
\end{table}

\begin{table}[t]
  \centering
  \caption{Stability of reduced models of the steel profile computed
    for varying orders $r = 5, 10, \ldots, 25$ and $N = 50$ quadrature nodes.
    Check marks indicate asymptotic stability, and all computed reduced-order
    models turned out to be asymptotically stable.}
  \label{tab:stabCheck}
  \vspace{.5\baselineskip}
  
  \begin{tabular}{lccccc}
    \hline\noalign{\smallskip} 
      & \multicolumn{1}{c}{$r = 5$}
      & \multicolumn{1}{c}{$r = 10$}
      & \multicolumn{1}{c}{$r = 15$}
      & \multicolumn{1}{c}{$r = 20$}
      & \multicolumn{1}{c}{$r = 25$}\\
      \noalign{\smallskip}\hline\noalign{\smallskip}
      \Quadbt{}
      & \checkmark
      & \checkmark
      & \checkmark
      & \checkmark
      & \checkmark\\
      \symQuadbt{}
      & \checkmark
      & \checkmark
      & \checkmark
      & \checkmark
      & \checkmark\\
      \FDsymQuadbt{}
      & \checkmark
      & \checkmark
      & \checkmark
      & \checkmark
      & \checkmark\\
      \noalign{\smallskip}\hline\noalign{\smallskip}
  \end{tabular}
\end{table}


\section{Conclusions}%
\label{sec:conclusions}

We presented a Hermite formulation of the quadrature-based balanced truncation
method that employs transfer function derivative data for the construction of
balanced reduced-order models.
The presented approach allows for the preservation of state-space Hermiticity,
dissipativity and, consequently, asymptotic stability independent of the
chosen reduction order and used quadrature points.
The numerical experiments verify the validity of the method and the theoretical
results, and showcase that also approximate derivative data can effectively be
used to learn dynamical systems with our method.


\phantomsection%
\section*{Acknowledgments}%
\addcontentsline{toc}{section}{Acknowledgments}

Parts of Werner's work were funded by the Deutsche Forschungsgemeinschaft
(DFG, German Research Foundation) under Germany's Excellence Strategy --
EXC-2047/2 -- 390685813.


\phantomsection%
\addcontentsline{toc}{section}{References}
\bibliographystyle{plainurl}
\bibliography{bibtex/myref}

\end{document}